\documentclass[a4paper,11pt]{amsart}
\usepackage{amsmath}
\usepackage{amsthm}
\usepackage{amssymb}
\usepackage{color}
\usepackage{graphicx}
\usepackage{tikz}
\usepackage{subfigure}
\usepackage{epstopdf}
\usepackage{cases}
\usepackage{bm}
\usepackage{euscript}
\newcommand{\R}{\mathbb{R}}
\newcommand{\D}{\mathbb{D}}
\newcommand{\C}{\mathbb{C}}
\renewcommand{\H}{\mathcal{H}}
\newcommand{\N}{\mathbb{N}}
\theoremstyle{plain}
\newtheorem{thm}{Theorem}[section]
\newtheorem*{thm*}{Theorem}
\newtheorem{lem}[thm]{Lemma}

\newtheorem{cor}[thm]{Corollary}
\theoremstyle{definition}

\newtheorem{rmk}[thm]{Remark}

\numberwithin{equation}{section}

\theoremstyle{letterthm}
\newtheorem{letterthm}{Theorem}

\newcommand{\abs}[1]{\left|#1\right|}
              \def\th{\theta}       
\def\vp{\varphi}      \def\z{\zeta}
\def\zth{\z_{\th}}  
\def\ra{\rightarrow}      \def \ov{\overline}
\def\beq{\begin{equation}}  \def\eeq{\end{equation}}
\def\beqq{\begin{equation*}}  \def\eeqq{\end{equation*}}
\def\bprof{\begin{proof}}    \def\eprof{\end{proof}}
\def\bad{\begin{aligned}}    \def\ead{\end{aligned}}
\def\l{\left}      \def\r{\right}
\def\bthm{\begin{thm}} \def\ethm{\end{thm}}
\def\blem{\begin{lem}} \def\elem{\end{lem}}
\def\p{\partial} \def\e{e^{i \th}}
\def\be{\begin{enumerate}} \def\ee{\end{enumerate}}

\begin{document}

\title[Variability regions for the fourth derivative]{Variability regions for the fourth derivative of bounded analytic functions}

\author{Gangqiang Chen}
\address{
School of Sciences,
Nanchang University,
Nanchang, China}
\email{cgqmath@qq.com; cgqmath@ims.is.tohoku.ac.jp}

\date{\today}

\begin{abstract}
Let $z_0$ and $w_0$ be given points in the open unit disk $\mathbb{D}$
with $|w_0| < |z_0|$, and  $\H_0$ be the class of all analytic self-maps $f$
of $\mathbb{D}$ normalized by $f(0)=0$.
In this paper, we establish the fourth-order Dieudonn\'e's Lemma and apply it to determine the variability region
$\{f^{(4)}(z_0): f\in \H_0,f(z_0) =w_0, f'(z_0)=w_1, f''(z_0)=w_2\}$ for given $z_0,w_0,w_1,w_2$ and give the form of all the extremal functions.
\end{abstract}

\maketitle

\section{Introduction}
\label{intro}
We denote by $\C$ the complex plane. For $c\in\C$ and $\rho>0$, we define the disks $\D(c, \rho)$ and $\overline{\D}(c, \rho)$ by
$\D(c, \rho):=\left\{ \zeta \in \C : |\zeta-c|< \rho \right\}$,
and
$\overline{\D}(c, \rho):=\left\{\zeta \in \C : |\zeta-c|\le \rho \right\}$. The open and closed unit disk $\D(0,1)$ are denoted by $\mathbb{D}$ and $\overline{\D}$ respectively. Throughout this article, let $z_0,w_0\in \D$ be given points with $|w_0|<|z_0|$, and $\H_0$ be the class of all analytic self-mappings $f$ of $\mathbb{D}$ satisfying $f(0)=0$. First we would like to recall the following classical result obtained by Schwarz in 1890, which describes the variability region of $f(z_0)$ for $z_0\in \D$ when $f$ ranges over $\H_0$.

\begin{letterthm}{\bf (Schwarz's Lemma)}
Let $z_0\in \mathbb{D}$. Then $\{f(z_0):f \in \mathcal{H}_0\}=\ov{\D}(0,\ |z_0|)$ and $\{f'(0):f \in \mathcal{H}_0\}=\ov{\D}$.\\
 Furthermore, $f(z_0)\in \partial \D(0,\ |z_0|)$ and $f'(0)\in \partial \D$ hold if and only if
 $ f(z)=e^{i \theta}z$ for some $\theta \in \mathbb{R}$.
\end{letterthm}

Since the discovery of Schwarz's Lemma, more and more authors considered the space $\H_0$ and the extensions of Schwarz's Lemma.
In 1931, Dieudonn\'e \cite{dieudonne1931} first described the variability region of $f'(z_0)$, $f\in \mathcal{H}_0$, at a fixed point $z_0\in \mathbb{D}$,
which could be considered as Schwarz's Lemma of the derivative.
For $a\in\D$, define $T_{a}\in \text{Aut}(\D)$ by
$$T_{a}(z)=\frac{z+a}{1+\overline{a}z},\quad z\in \D,$$
and $c_1(z_0,w_0), \rho_1(z_0,w_0)$ by
\begin{equation*}
\left\{
\begin{aligned}
&c_1(z_0,w_0)=\frac{w_0}{z_0},\\
&\rho_1(z_0,w_0)=\frac{|z_0|^2-|w_0|^2}{|z_0|(1-|w_0|^2)},
\end{aligned}
\right.
\end{equation*}
then we show his observation as follows(see also \cite{beardon2004multi}, \cite{chen_2019} and \cite{rivard2013application}). .
\begin{letterthm}{\bf (Dieudonn\'e's Lemma)}
Let $z_0,w_0\in \D$ and $|w_0|<|z_0|$. Then
$\{f'(z_0):f \in \mathcal{H}_0, f(z_0)=w_0\}= \overline{\mathbb{D}}
  \left( c_1(z_0,w_0) , \rho_1(z_0,w_0) \right)$. \\
  Furthermore,
$f'(z_0) \in \partial \D \left( c_1(z_0,w_0) , \rho_1(z_0,w_0) \right) $ for $\theta\in \R$ if and only if
$f(z)=zT_{w_0/z_0}(e^{i \theta}T_{-z_0}(z))$.
\end{letterthm}

In 1934, Rogosinski\cite{rogosinski1934} determined the variability region of $f(z)$ for $z\in\D$, $f\in \H_0$ with $|f'(0)|<1$, which can be considered as a
sharpened version of Schwarz's Lemma (see also \cite{duren1983univalent} and \cite{goluzin1969geometric}). is an improvement of the derivative part of
\begin{letterthm}{\bf (Rogosinski's Lemma)}
If $f\in \H_0$ and $f'(0)$ is fixed, then for $z\in \D\setminus\{0\}$, the region of values of $f(z)$ is the closed disk $\ov{\D}(c,r)$, where
 $$c=\dfrac{zf'(0)(1-z^2)}{1-|z|^2|f'(0)|^2},\quad r=|z|^2 \dfrac{1-|f'(0)|^2}{1-|z|^2|f'(0)|^2}.$$
\end{letterthm}

In 1996, Mercer \cite{mercer1997sharpened} obtained a description of the variability region of $f(z)$ for $z\in\D$, $f\in \H_0$ with $f(z_0)=w_0 (z_0\ne 0)$. It is worth mentioning that Rogosinski's Lemma and Dieudonn\'e's Lemma are the limiting cases of Mercer's result.

In recent years, a lot of the articles on regions of variability have been written \cite{Ponnusamy2007univalent,Ponnusamy2008close-to-convex,Ponnusamy2009satisfying,yanagihara2005}. Among others,
Rivard \cite{rivard2013application} obtained the so-called second-order Dieudonn\'e Lemma which demonstrates that
if $f\in\mathcal{H}_0$  is not an automorphism of $\mathbb{D}$, then
   \begin{align}\label{ineq:f''}
       &\left|\frac{1}{2}z_0^2 f''(z_0)-\frac{z_0 w_1-w_0}{1-|z_0|^2}
         +\frac{\overline{w_0}(z_0 w_1-w_0)^2}{|z_0|^2-|w_0|^2}\right| +\frac{|z_0||z_0 w_1-w_0|^2}{|z_0|^2-|w_0|^2}\nonumber\\
       &\qquad \le \frac{|z_0| (|z_0|^2-|w_0|^2)}{(1-|z_0|^2)^2},
   \end{align}
where $f(z_0)=w_0$ and $f'(z_0)=w_1\in \Delta(z_0,w_0)$. Equality in \eqref{ineq:f''} holds if and only if $f(z)=zg(z)$ where $g(z)$ is a Blaschke product of degree 1 or 2 (see also \cite{cho2012multi}). Here, we remark that a function $B(z)$ is called a Blaschke product of degree $n \in \N$  if it takes the form
         $$ B(z)=e^{i \theta}\prod\limits_{j=1}^{n}
         \frac{z-z_j}{1-\overline{z_j}z}, \quad z, z_j\in \D, \theta \in \mathbb{R}.$$
We appropriately modify Rivard's result as follows (see \cite{chen_2019}). Denote $c_2(z_0,w_0,\lambda)$ and $\rho_2(z_0,w_0,\lambda)$ by
\begin{equation*}
\left\{
\begin{aligned}
c_2(z_0,w_0,\lambda)&=\frac{2(r^2-s^2)}{z_0^2(1-r^2)^2}\lambda(1-\overline{w_0}\lambda),\\ \rho_2(z_0,w_0,\lambda)&=\frac{2(r^2-s^2)}{r(1-r^2)^2}(1-|\lambda|^2).
\end{aligned}
\right.
\end{equation*}

\begin{letterthm}{\bf (The second-order Dieudonn\'e's Lemma)}
Let $z_0, w_0\in \mathbb{D}$, $\lambda\in \ov{\D}$ with $|w_0|=s<r=|z_0|$, 
$$w_1= c_1(z_0,w_0)+\rho_1(z_0,w_0)\dfrac{r\lambda}{z_0}.$$
 Suppose that $f\in\mathcal{H}_0$, $f(z_0) = w_0$ and $f'(z_0)=w_1$.
Set $u_0=w_0/z_0$ and $\lambda_0=r^2\lambda/z_0^2$. Then
\begin{enumerate}
\item If $|\lambda|=1$, then $f''(z_0)=c_2(z_0,w_0,\lambda)$ and $f(z)=z T_{u_0}(\lambda_0 T_{-z_0}(z))$.
\item If $|\lambda|<1$, then the region of values of $f'''(z_0)$ is the closed disk
\begin{align*}
&\overline{\D}(c_2(z_0,w_0,\lambda), \rho_2(z_0,w_0,\lambda))\\
&=\{z T_{u_0}\left(T_{-z_0}(z) T_{\lambda_0}(\alpha T_{-z_0}(z))\right):\alpha\in \overline{\D}\}.
\end{align*}
Furthermore, $f''(z_0)\in \p\D(c_2(z_0,w_0,\lambda),\rho_2(z_0,w_0,\lambda))$ if and only if\\
$f(z)=z T_{u_0}\left(T_{-z_0}(z) T_{\lambda_0}(\e T_{-z_0}(z))\right)$, where $\theta \in \R$.
\end{enumerate}
\end{letterthm}

It is worth mentioning that the present author and Yanagihara \cite{chen2020} applied this consequence to precisely determine the variability region
$V(z_0,w_0)=\{f''(z_0):f\in\H_0, f(z_0)=w_0\}$.

In 2021, the present author \cite{chen2021} obtained the third-order Dieudonn\'e's Lemma as follows.
Denote $c_3(z_0,w_0,\lambda,\mu)$ and $\rho_3(z_0,w_0,\lambda,\mu)$ by
\begin{equation*}
\left\{
\begin{aligned}
c_3(z_0,w_0,\lambda,\mu)
&=\frac{6(r^2-s^2)}{z_0^3(1-r^2)^3}\left(\mathcal{A}+z_0\mu(1-|\lambda|^2)(1+r^2-2 \ov{w_0}\lambda-z_0 \ov{\lambda}\mu)\right);\\
\rho_3(z_0,w_0,\lambda,\mu)&=\frac{6(r^2-s^2)}{r(1-r^2)^3}(1-|\lambda|^2)(1-|\mu|^2),
\end{aligned}
\right.
\end{equation*}
where
$$\mathcal{A}=\ov{w_0}^2 \lambda^3-\ov{w_0}(1+r^2)\lambda^2+r^2\lambda.$$
\begin{letterthm}{\bf (The third-order Dieudonn\'e's Lemma)}
Let $z_0, w_0\in \mathbb{D}$, $\lambda, \mu \in \ov{\D}$ with $|w_0|=s<r=|z_0|$,
\begin{equation*}
\left\{
\begin{aligned}
w_1&= c_1(z_0,w_0)+\rho_1(z_0,w_0)\dfrac{r\lambda}{z_0};\\
w_2&=c_2(z_0,w_0,\lambda)+\rho_2(z_0,w_0,\lambda)\frac{r\mu}{z_0}.
\end{aligned}
\right.
\end{equation*}
 Suppose that $f\in\mathcal{H}_0$, $f(z_0) = w_0$, $f'(z_0)=w_1$ and $f''(z_0)=w_2$.
Set $u_0=w_0/z_0$, $\lambda_0=r^2\lambda/z_0^2$ and $\mu_0=r^2\mu/z_0^2$.
\begin{enumerate}
\item If $|\lambda|=1$, then $f'''(z_0)=c_3(z_0,w_0,\lambda,\mu)$ and $f(z)=z T_{u_0}(\lambda_0 T_{-z_0}(z))$.
\item If $|\lambda|<1$, $|\mu|=1$, then $f'''(z_0)=c_3(z_0,w_0,\lambda,\mu)$ and $f(z)=z T_{u_0}\l(T_{-z_0}(z) T_{\lambda_0}(\mu_0 T_{-z_0}(z))\r)$.
\item If $|\lambda|<1$, $|\mu|<1$, then the region of values of $f'''(z_0)$ is the closed disk
\begin{align*}
&\overline{\D}(c_3(z_0,w_0,\lambda,\mu), \rho_3(z_0,w_0,\lambda,\mu))\\
&=\{z T_{u_0}\left(T_{-z_0}(z) T_{\lambda_0}(T_{-z_0}(z) T_{\mu_0}(\alpha T_{-z_0}(z)))\right):\alpha\in \overline{\D}\}.
\end{align*}
Furthermore, $f'''(z_0)\in \p\D(c, \rho)$ if and only if\\
$f(z)=z T_{u_0}\left(T_{-z_0}(z) T_{\lambda_0}(T_{-z_0}(z) T_{\mu_0}(\e T_{-z_0}(z)))\right)$, where $\theta \in \R$.
\end{enumerate}
\end{letterthm}
Before the statement of our main result,
we denote $c_4(z_0,w_0,\lambda,\mu,\tau)$ and $\rho_4(z_0,w_0,\lambda,\mu,\tau)$ by
\begin{small}
\begin{equation}
\left\{
\begin{aligned}
c_4(z_0,w_0,\lambda,\mu,\tau)&=\frac{24(r^2-s^2)}{z_0^4(1-r^2)^4}[\mathcal{B}+z_0^2\tau(1-|\lambda|^2)(1-|\mu|^2)(1+2r^2-2\overline{w_0}\lambda-2z_0\bar{\lambda}\mu-z_0\bar{\mu}\tau)];\\
\rho_4(z_0,w_0,\lambda,\mu,\tau)&=\frac{24(r^2-s^2)}{r(1-r^2)^4}(1-|\lambda|^2)(1-|\mu|^2)(1-|\tau|^2),
\end{aligned}
\right.
\end{equation}
\end{small}
where
\begin{small}
\begin{equation}\label{eq:AB}
\begin{aligned}
\mathcal{B}&=\lambda r^6-\overline{w_0}^3\lambda^4-3\overline{w_0}^2\lambda^2(-\overline{w_0}\lambda^2+r^2\lambda+z_0\mu(1-|\lambda|^2))\\
&\quad +(1-r^2-2\overline{w_0}\lambda)(\lambda(\overline{w_0}\lambda-r^2)^2+z_0\mu(1-|\lambda|^2)(2r^2-2\overline{w_0}\lambda-z_0\bar{\lambda}\mu))\\
&\quad-\overline{w_0}(-\overline{w_0}\lambda^2+r^2\lambda+z_0\mu(1-|\lambda|^2))^2+z_0^3(1-|\lambda|^2)(\overline{\lambda}^2\mu^3+3\overline{z_0}^2\mu-3\overline{z_0}\overline{\lambda}\mu^2)
\end{aligned}
\end{equation}
\end{small}

\begin{thm}[ The fourth-order  Dieudonn\'e's Lemma]\label{thm:lemfourth}
Let $z_0, w_0\in \mathbb{D}$, $\lambda, \mu, \tau \in \ov{\D}$ with $|w_0|=s<r=|z_0|$,
\begin{equation*}
\left\{
\begin{aligned}
w_1&= c_1(z_0,w_0)+\rho_1(z_0,w_0)\dfrac{r\lambda}{z_0};\\
w_2&=c_2(z_0,w_0,\lambda)+\rho_2(z_0,w_0,\lambda)\frac{r\mu}{z_0};\\
w_3&=c_3(z_0,w_0,\lambda,\mu)+\rho_3(z_0,w_0,\lambda,\mu)\frac{r\tau}{z_0}.
\end{aligned}
\right.
\end{equation*}
Suppose that $f\in\mathcal{H}_0$, $f(z_0) = w_0$, $f'(z_0)=w_1$, $f''(z_0)=w_2$, $f'''(z_0)=w_3$.
Set $u_0=w_0/z_0$, $\lambda_0=r^2\lambda/z_0^2$, $\mu_0=r^2\mu/z_0^2$ and $\tau_0=r^2\tau/z_0^2$.
\begin{enumerate}
\item If $|\lambda|=1$, then $f^{(4)}(z_0)=c_4(z_0,w_0,\lambda,\mu,\tau)$ and $f(z)=z T_{u_0}(\lambda_0 T_{-z_0}(z))$.

\item If $|\lambda|<1$, $|\mu|=1$, then $f^{(4)}(z_0)=c_4(z_0,w_0,\lambda,\mu,\tau)$ and $f(z)=z T_{u_0}\l(T_{-z_0}(z) T_{\lambda_0}(\mu_0 T_{-z_0}(z))\r)$.
\item If $|\lambda|<1$, $|\mu|<1$, $|\tau|=1$, then $f^{(4)}(r)=c_4(z_0,w_0,\lambda,\mu,\tau)$ and
    $f(z)=z T_{u_0}\left(T_{-z_0}(z) T_{\lambda_0}(T_{-z_0}(z) T_{\mu_0}(\tau_0 T_{-z_0}(z)))\right)$.
\item If $|\lambda|<1$, $|\mu|<1$, $|\tau|<1$, then the region of values of $f^{(4)}(z_0)$ is the closed disk
\begin{small}
\begin{align*}
&\overline{\D}(c_4(z_0,w_0,\lambda,\mu,\tau), \rho_4(z_0,w_0,\lambda,\mu,\tau))\\
&=\{z T_{u_0}\left(T_{-z_0}(z) T_{\lambda_0}(T_{-z_0}(z) T_{\mu_0}(T_{-z_0}(z)T_{\tau_0 }(\alpha T_{-z_0}(z))))\right):\alpha\in\overline{\D} \}.
\end{align*}
\end{small}
Furthermore, $f^{(4)}(z_0)\in \p\D(c_4, \rho_4)$ if and only if\\
$f(z)=z T_{u_0}\left(T_{-z_0}(z) T_{\lambda_0}(T_{-z_0}(z) T_{\mu_0}(T_{-z_0}(z)T_{\tau_0 }(e^{i \theta} T_{-z_0}(z))))\right)$,where $\theta \in \R$.
\end{enumerate}
\end{thm}
Naturally, we shall further study the fourth order derivative $f^{(4)}$ of $f\in \H_0$,which leads to establishing a fourth-order Dieudonn\'e's Lemma, then apply our result to determine the region of values of $f^{(4)}(z_0)$, $f\in\H_0$, in terms of $z_0, f(z_0),f'(z_0),f''(z_0),f'''(z_0)$, and give the form of all the extremal functions.
We believe that the study on the fourth derivatives of bounded analytic functions could serve as a basis for further investigations on the variability regions of higher derivatives.

\section{Proof of the fourth-order Dieudonn\'e's Lemma}
We begin this section with the introduction to Peschl's invariant derivatives.
For $g:\mathbb{D}\to \mathbb{D}$ holomorphic, the so-called Peschl's invariant derivatives $D_n g(z)$ are defined by the Taylor series expansion (see \cite{peschl1955invariants}):
$$z\ra h(z):=\frac{g(\frac{z+z_0}{1+\overline{z}_0 z})-g(z_0)}{1-\overline{g(z_0)}g(\frac{z+z_0}{1+\overline{z}_0 z})}=\sum_{n=1}^{\infty}\frac{D_n g(z_0)}{n!}z^n,\quad z, z_0\in \D,$$
where $D_n g(z_0)=h^{(n)}(0)$.

Precise forms of $D_n g(z)$, $n=1,2,3,4$, are expressed by
\begin{small}
\begin{align*}
D_1 g(z)&=\frac{(1-\abs{z}^2)g'(z)}{1-\abs{g(z)}^2},\\
D_2 g(z)&=\frac{(1-\abs{z}^2)^2}{1-\abs{g(z)}^2}
\Bigg[g''(z)-\frac{2\overline{z}g'(z)}{1-\abs{z}^2}
+\frac{2\overline{g(z)}g'(z)^2} {1-\abs{g(z)}^2}\Bigg],\\
D_3 g(z)&=\frac{(1-\abs{z}^2)^3}{1-\abs{g(z)}^2}
\Bigg[g'''(z)-\frac{6\overline{z}g''(z)}{1-\abs{z}^2}
+\frac{6\overline{g(z)}g'(z)g''(z)} {1-\abs{g(z)}^2}
+\frac{6\overline{z}^2g'(z)}{(1-\abs{z}^2)^2}\\
&\qquad\qquad -\frac{12\overline{g(z)}g'(z)^2} {(1-\abs{z}^2)(1-\abs{g(z)}^2)}
+\frac{6\overline{g(z)}^2g'(z)^3} {(1-\abs{g(z)}^2)^2}\Bigg],\\
D_4 g(z)&=\frac{(1-|z|^2)^4}{1-|g(z)|^2}\Bigg[ g^{(4)}(z)-\frac{12\bar{z} g^{(3)}(z)}{1-|z|^2}+\frac{6 \overline{g(z)} g''(z)^2}{1-|g(z)|^2}+\frac{36\bar{z}^2 g''(z)}{(1-|z|^2)^2}\\
&\qquad\qquad\qquad+\frac{24 \overline{g(z)}^3 g'(z)^4}{(1-|g(z)|^2)^3}-\frac{72\bar{z}\overline{g(z)}^2 g'(z)^3}{(1-|z|^2)(1-|g(z)|^2)^2}+\frac{72\bar{z}^2 \overline{g(z)} g'(z)^2}{(1-|z|^2)^2(1-|g(z)|^2)}\\
&\qquad-\frac{24 \bar{z}^3 g'(z)}{(1-|z|^2)^3}+\frac{8\overline{g(z)} g'(z)g^{(3)}(z) }{1-|g(z)|^2}+\frac{36 \overline{g(z)}^2 g'(z)^2 g''(z)}{(1-|g(z)|^2)^2}-\frac{72\bar{z} \overline{g(z)} g'(z) g''(z)}{(1-|z|^2)(1-|g(z)|^2)}\Bigg].
\end{align*}
\end{small}
In 2012, Cho, Kim and Sugawa \cite{cho2012multi} obtained the following inequality in terms of Peschl's
invariant derivatives, we shall interpret it as an inequality for $g^{(4)}(z)$ in terms of $z, g(z), g'(z),  g''(z)$ and $g'''(z)$.
\begin{lem}[\cite{cho2012multi}]\label{lem:cho}
If $g:\mathbb{D}\to \mathbb{D}$ is holomorphic, then
\begin{equation}
\begin{aligned}
\frac{D_4 g(z)}{24}\left[(1-|D_1 g(z)|^2)^2-\left|\frac{D_2 g(z)}{2}\right|^2\right]&+2\overline{D_1 g(z)}\frac{D_2 g(z)}{2}\frac{D_3 g(z)}{6}(1-|D_1 g(z)|^2)\\
+\overline{D_1 g(z)}^2\left(\frac{D_2 g(z)}{2}\right)^3+&\overline{\frac{D_2 g(z)}{2}}\left(\frac{D_3 g(z)}{6}\right)^2\\
\leq(1-|D_1 g(z)|^2)^3-(1-|D_1 g(z)|^2)&\left(\left|\frac{D_3 g(z)}{6}\right|^2+2\left|\frac{D_2 g(z)}{2}\right|^2\right)+\left|\frac{D_2 g(z)}{2}\right|^4\\
-D_1 g(z)\left(\overline{\frac{D_2 g(z)}{2}}\right)^2&\frac{D_3 g(z)}{6}-\overline{D_1 g(z)}\left(\frac{D_2 g(z)}{2}\right)^2\overline{\frac{D_3 g(z)}{6}},
\end{aligned}
\end{equation}
equality holds for a point $z\in \mathbb{D}$ if and only if $g$ is a Blaschke product of degree at most 4.
\end{lem}

To simplify the proof of Theorem \ref{thm:lemfourth}, we consider the following relations. Assume that $z_0=re^{i \vp}, w_0=s e^{i \xi}\in \D$ with $s<r$, define
the `rotation function' $\tilde{f}$ by
$\tilde{f}(z)=e^{-i \xi}f(e^{i \vp} z)$, then we obtain $\tilde{f}'(r)=e^{i(\vp-\xi)}f'(z_0)$, $\tilde{f}''(r)=e^{i(2\vp-\xi)}f''(z_0)$ and $\tilde{f}'''(r)=e^{i(3\vp-\xi)}f'''(z_0)$, $\tilde{f}^{(4)}(r)=e^{i(4\vp-\xi)}f^{(4)}(z_0)$.
Therefore, we can relabel $\tilde{f}$ as $f$, and assume that
\begin{align*}
z_0&=r,w_0=s,\\
w_1&=c_1(r,s)+\rho_1(r,s)\lambda, \quad \lambda\in \ov{\D},\\
w_2&=c_2(r,s,\lambda)+\rho_2(r,s,\lambda)\mu,\quad \mu\in \ov{\D},\\
w_3&=c_3(r,s,\lambda,\mu)+\rho_2(r,s,\lambda,\mu)\tau,\quad \tau\in \ov{\D}.
\end{align*}
Correspondingly,
$c_4(r,s,\lambda,\mu,\tau)$ and $\rho_4(r,s,\lambda,\mu,\tau)$ are denoted by
\begin{equation}
\left\{
\begin{aligned}
c_4(r,s,\lambda,\mu,\tau)&=A[B+r^2\tau(1-|\lambda|^2)(1-|\mu|^2)(1+2r^2-2s\lambda-2r\bar{\lambda}\mu-r\bar{\mu}\tau)];\\
\rho_4(r,s,\lambda,\mu,\tau)&=Ar^3(1-|\lambda|^2)(1-|\mu|^2)(1-|\tau|^2),
\end{aligned}
\right.
\end{equation}
where
\begin{equation}\label{eq:AB}
\left\{
\begin{aligned}
A&=\frac{24(r^2-s^2)}{r^4(1-r^2)^4},\\
B&=\lambda r^6-s^3\lambda^4-3s^2\lambda^2(-s\lambda^2+r^2\lambda+r\mu(1-|\lambda|^2))\\
&\quad +(1-r^2-2s\lambda)(\lambda(s\lambda-r^2)^2+r\mu(1-|\lambda|^2)(2r^2-2s\lambda-r\bar{\lambda}\mu))\\
&\quad-s(-s\lambda^2+r^2\lambda+r\mu(1-|\lambda|^2))^2+r^3(1-|\lambda|^2)(\overline{\lambda}^2\mu^3+3r^2\mu-3r\overline{\lambda}\mu^2)
\end{aligned}
\right.
\end{equation}
Assume that $g(z)=f(z)/z$, then $g$ is an analytic self-mapping of $\D$. A straight computation shows that
$D_1 g(r)=\lambda$, $D_2 g(r)=2\mu(1-|\lambda|^2)$ and
$D_3 g(r)=6(1-|\lambda|^2)\left[-\overline{\lambda}\mu^2+\tau(1-|\mu|^2)\right]$.
From Lemma \ref{lem:cho}, we have
$$\left|\frac{D_4 g(r)}{24}+(1-|\lambda|^2)\left[-\overline{\lambda}^2\mu^3+(1-|\mu|^2)(2\overline{\lambda}\mu\tau+\overline{\mu}\tau^2)\right]\right|
\leq(1-|\lambda|^2)(1-|\mu|^2)(1-|\tau|^2),$$

Then we obtain
\begin{equation}\label{eq:f''''(r)}
|f^{(4)}(r)-c_4(r,s,\lambda,\mu,\tau)|\le \rho_4(r,s,\lambda,\mu,\tau).
\end{equation}
Equality in \eqref{eq:f''''(r)} holds if and only if $f(z)=zg(z)$, where
$g$ is a Blaschke product of degree 1, 2, 3 or 4 and satisfies

\begin{equation}\label{condition2}
\left\{
\begin{aligned}
g(r)&=\frac{s}{r};\\
g'(r)&=\frac{r^2-s^2}{r^2(1-r^2)}\lambda;\\
g''(r)&=\frac{2(r^2-s^2)}{r^3(1-r^2)^2}(-s\lambda^2+r^2\lambda+r\mu(1-|\lambda|^2));\\
g'''(r)&=\frac{6(r^2-s^2)}{r^4(1-r^2)^3}
\left[b+r^2\tau(1-|\lambda|^2)(1-|\mu|^2)\right].
\end{aligned}
\right.
\end{equation}
where
$$b=\lambda(s\lambda-r^2)^2+r\mu(1-|\lambda|^2)(2r^2-2s\lambda-r\bar{\lambda}\mu).$$
Then Theorem \ref{thm:lemfourth} is reduced to the following corollary.

\begin{cor}\label{cor:third}
Let $0\le s<r<1$, $\lambda, \mu, \tau \in \ov{\D}$ with
\begin{equation*}
\left\{
\begin{aligned}
w_1&=c_1(r,s)+\rho_1(r,s)\lambda, \\ w_2&=c_2(r,s,\lambda)+\rho_2(r,s,\lambda)\mu,\\
w_3&=c_3(r,s,\lambda,\mu)+\rho_3(r,s,\lambda,\mu)\tau.
\end{aligned}
\right.
\end{equation*}
 Suppose that $f\in\mathcal{H}_0$, $f(r)= s$, $f'(r)=w_1$, $f''(r)=w_2$ and $f'''(r)=w_3$. Set $u_0=s/r$.
\begin{enumerate}
\item If $|\lambda|=1$, then $f^{(4)}(r)=c_4(r,s,\lambda,\mu,\tau)$ and $f(z)=z T_{u_0}(\lambda T_{-r}(z))$.

\item If $|\lambda|<1$, $|\mu|=1$, then $f^{(4)}(r)=c_4(r,s,\lambda,\mu,\tau)$ and $f(z)=z T_{u_0}\l(T_{-r}(z) T_{\lambda}(\mu T_{-r}(z))\r)$.

\item If $|\lambda|<1$, $|\mu|<1$, $|\tau|=1$, then $f^{(4)}(r)=c_4(r,s,\lambda,\mu,\tau)$ and $f(z)=
    z T_{u_0}\left(T_{-r}(z) T_{\lambda}(T_{-r}(z) T_{\mu}(\tau T_{-r}(z)))\right)$.

\item If $|\lambda|<1$, $|\mu|<1$, $|\tau|<1$, then the region of values of $f^{(4)}(z_0)$ is the closed disk
\begin{align*}
\qquad&\overline{\D}(c_4(r,s,\lambda,\mu,\tau), \rho_4(r,s,\lambda,\mu,\tau))\\
&=\{z T_{u_0}\left(T_{-r}(z) T_{\lambda}(T_{-r}(z) T_{\mu}(T_{-r}(z) T_{\tau}(\alpha T_{-r}(z))))\right):\alpha\in \overline{\D}\}.
\end{align*}
Furthermore, $f^{(4)}(r)\in \p\D(c_4(r,s,\lambda,\mu,\tau), \rho_4(r,s,\lambda,\mu,\tau))$ if and only if
$f(z)=z T_{u_0}\left(T_{-r}(z) T_{\lambda}(T_{-r}(z) T_{\mu}(T_{-r}(z) T_{\tau}(e^{i \theta} T_{-r}(z))))\right)$, where $\theta \in \R$.
\end{enumerate}
\end{cor}
\begin{proof}
We can easily prove Case (1), (2) and (3) by using the same method in the proof of \cite[Lemma 2.2]{chen_2019} or \cite[Theorem 1.1]{chen2021}.

For Case (4), the inequality \eqref{eq:f''''(r)} means that 
$$f^{(4)}(r) \in \overline{\D}(c_4(r,s,\lambda,\mu,\tau),\rho_4(r,s,\lambda,\mu,\tau)).$$
 To show that $\overline{\D}(c_4(r,s,\lambda,\mu,\tau),\rho_4(r,s,\lambda,\mu,\tau))$ is covered, let $\alpha \in \ov{\D}$, $u_0=s/r$ and set
 $f(z)=zg(z)$, where
 $$g(z)=T_{u_0}\left(T_{-r}(z) T_{\lambda}(T_{-r}(z) T_{\mu}(T_{-r}(z) T_{\tau}(\alpha T_{-r}(z))))\right).$$
Then $f(0)=0$ and $f(r)=s$. Next we need to show that $f'(r)=w_1$.
Let $u(z)=T_{-r}(z) T_{\mu}(v(z))$, $v(z)=T_{-r}(z) T_{\tau}(w(z))$, $w(z)=\alpha T_{-r}(z)$, then we have
\begin{equation}\label{eq:T-g}
T_{-u_0}\circ g(z)=T_{-r}(z) T_{\lambda}(u(z)).
\end{equation}
Differentiating both sides and using the chain rule, we get
\begin{equation}\label{g'}
\begin{aligned}
(T_{-u_0})'(g(z))g'(z)&=T_{-r}'(z) T_{\lambda}(u(z))
+T_{-r}(z) T_{\lambda}'(u(z))u'(z)
\end{aligned}
\end{equation}
for all $z\in \D$.
Substituting $z=r$ into this equation, we have
$$(T_{-u_0})'(g(r))g'(r)=T_{-r}'(r) T_{\lambda}(0),$$
which implies
$$g'(r)=\frac{(r^2-s^2)\lambda}{r^2(1-r^2)}.$$
Thus, we obtain that $f$ satisfies
$$f'(r)=g(r)+rg'(r)=w_1.$$
Similarly, differentiating both sides of \eqref{g'}, we obtain
\begin{equation}\label{appendix:g''}
\begin{aligned}
&(T_{-u_0})''(g(z)) (g'(z))^2+(T_{-u_0})'(g(z)) g''(z)\\
&=T_{-r}''(z) T_{\lambda}(u(z))+2 T_{-r}'(z) T_{\lambda}'(u(z)) u'(z)\\
&\quad+T_{-r}(z) \left(T_{\lambda}''(u(z))u'(z)^2 +T_{\lambda}'(u(z))u''(z) \right),\quad z\in \D.
\end{aligned}
\end{equation}
Substituting $z=r$ into the above equation,
\begin{align*}
&(T_{-u_0})''(g(r)) (g'(r))^2+(T_{-u_0})'(g(r))g''(r)\\
&=T_{-r}''(r) T_{\lambda}(0)+2T_{-r}'(r) T_{\lambda}'(0)u'(r).
\end{align*}
We get that
$$g''(r)=\frac{2(r^2-s^2)}{r^3(1-r^2)^2}(-s\lambda^2+r^2\lambda+r\mu(1-|\lambda|^2)).$$
The above with $
f''(z)=2g'(z)+z g''(z)$ immediately yields $f''(r)=w_2$.

Next we compute the value of $f'''(r)$. Differentiating both sides of \eqref{appendix:g''},
\begin{equation}\label{appendix:g'''}
\begin{aligned}
&(T_{-u_0})'''(g(z)) (g'(z))^3+3(T_{-u_0})''(g(z))g'(z) g''(z)+T_{-u_0}'(g(z)) g'''(z)\\
&=T_{-r}^{(3)}(z) T_{\lambda}(u(z))+3 T_{-r}''(z) T_{\lambda}'(u(z)) u'(z)+3 T_{-r}'(z) \left( T_{\lambda}''(u(z))u'(z)^2+T_{\lambda}'(u(z))u''(z)\right)\\
&+T_{-r}(z) \left(T_{\lambda}^{(3)}(u(z)) u'(z)^3+3 T_{\lambda}''(u(z))u'(z) u''(z) +T_{\lambda}'(u(z))u^{(3)}(z) \right)\\
\end{aligned}
\end{equation}
 and then
substituting $z=r$ into \eqref{appendix:g'''}, we have
\begin{align*}
&(T_{-u_0})'''(g(r)) (g'(r))^3+3(T_{-u_0})''(g(r)) g''(r)+T_{-u_0}'''(g(r)) g'''(r)\\
&\quad=T_{-r}'''(r) T_{\lambda}(0)+3T_{-r}''(r) T_{\lambda}'(0)u'(r)+3T_{-r}'(r) \left(T_{\lambda}''(0)u'(r)^2+ T_{\lambda}'(0)u''(r)\right).
\end{align*}
We get $$g'''(r)=\frac{6(r^2-s^2)}{r^4(1-r^2)^3}
\left[\lambda(s\lambda-r^2)^2+r\mu(1-|\lambda|^2)(2r^2-2s\lambda-r\bar{\lambda}\mu)+r^2\tau(1-|\lambda|^2)(1-|\mu|^2)\right].$$
Together with $f'''(z)=3g''(z)+z g'''(z)$, we obtain
$
f'''(r)=w_3$.

It remains to determine the form of $f^{(4)}(r)$. Differentiating both sides of \eqref{appendix:g'''}, we have
\begin{equation}\label{appendix:g''''}
\begin{aligned}
&g^{(4)}(z) (T_{-u_0})'(g(z))+(T_{-u_0})^{(4)}(g(z)) g'(z)^4+3  (T_{-u_0})''(g(z)) g''(z)^2\\
&\qquad\qquad\qquad\qquad\quad+4 g^{(3)}(z) g'(z) (T_{-u_0})''(g(z))+6 (T_{-u_0})^{(3)}(g(z)) g'(z)^2 g''(z)\\
&=T_{-r}^{(4)}(z) T_{\lambda}(u(z))+4 T_{-r}^{(3)}(z) T_{\lambda}'(u(z)) u'(z)+6 T_{-r}''(z) \left(T_{\lambda}''(u(z)) u'(z)^2 + T_{\lambda}'(u(z))u''(z)\right)\\
&\quad+4 T_{-r}'(z) \left(T_{\lambda}^{(3)}(u(z)) u'(z)^3+3T_{\lambda}''(u(z)) u'(z) u''(z) +T_{\lambda}'(u(z))u^{(3)}(z) \right)\\
&\quad+T_{-r}(z) \big(T_{\lambda}^{(4)}(u(z)) u'(z)^4+6 T_{\lambda}^{(3)}(u(z)) u'(z)^2 u''(z) \\
&\qquad\qquad\qquad+3T_{\lambda}''(u(z)) u''(z)^2+4 T_{\lambda}''(u(z)) u'(z) u^{(3)}(z) +T_{\lambda}'(u(z))u^{(4)}(z) \big)
\end{aligned}
\end{equation}
and then
substituting $z=r$ into \eqref{appendix:g''''}, we have
\begin{align*}
&g^{(4)}(r) (T_{-u_0})'(g(r))+(T_{-u_0})^{(4)}(g(r)) g'(r)^4+3  (T_{-u_0})''(g(r)) g''(r)^2\\
&\qquad\qquad\qquad\qquad\quad+4 g^{(3)}(r) g'(r) (T_{-u_0})''(g(r))+6 (T_{-u_0})^{(3)}(g(r)) g'(r)^2 g''(r)\\
&=T_{-r}^{(4)}(r) T_{\lambda}(0)+4 T_{-r}^{(3)}(r) T_{\lambda}'(0) u'(r)+6 T_{-r}''(r) \left(T_{\lambda}''(0) u'(r)^2 + T_{\lambda}'(0)u''(r)\right)\\
&\quad+4 T_{-r}'(r) \left(T_{\lambda}^{(3)}(0) u'(r)^3+3T_{\lambda}''(0) u'(r) u''(r) +T_{\lambda}'(0)u^{(3)}(r) \right).
\end{align*}
We get
\begin{align*}
&g^{(4)}(r)=\frac{24(r^2-s^2)}{r^5(1-r^2)^4}\left[\lambda r^6-s^3\lambda^4-3s^2\lambda^2(\lambda(r^2-s\lambda)+r\mu(1-|\lambda|^2))\right.\\
&-2s\lambda(\lambda(s\lambda-r^2)^2+r\mu(1-|\lambda|^2)(2r^2-2s\lambda-r\bar{\lambda}\mu)+r^2\tau(1-|\lambda|^2)(1-|\mu|^2))\\
&-s(\lambda(r^2-s\lambda)+r\mu(1-|\lambda|^2))^2+r^3(1-|\lambda|^2)(\overline{\lambda}^2\mu^3+3r^2\mu-3r\overline{\lambda}\mu^2)\\
&\left.+r^3(1-|\lambda|^2)(1-|\mu|^2)(3r\tau-2\bar{\lambda}\mu\tau-\bar{\mu}\tau^2)+r^3(1-|\lambda|^2)(1-|\mu|^2)(1-|\tau|^2)\alpha\right].
\end{align*}
Together with $f^{(4)}(z)=4g'''(z)+z g^{(4)}(z)$, we obtain
\begin{align*}
f^{(4)}(r)= &rg^{(4)}(r)+4g'''(r)=\frac{24(r^2-s^2)}{r^4(1-r^2)^4}\left[\lambda r^6-s^3\lambda^4-3s^2\lambda^2(-s\lambda^2+r^2\lambda+r\mu(1-|\lambda|^2))\right.\\
&+(1-r^2-2s\lambda)(\lambda(s\lambda-r^2)^2+r\mu(1-|\lambda|^2)(2r^2-2s\lambda-r\bar{\lambda}\mu)+r^2\tau(1-|\lambda|^2)(1-|\mu|^2))\\
&-s(-s\lambda^2+r^2\lambda+r\mu(1-|\lambda|^2))^2+r^3(1-|\lambda|^2)(\overline{\lambda}^2\mu^3+3r^2\mu-3r\overline{\lambda}\mu^2)\\
&+r^3(1-|\lambda|^2)(1-|\mu|^2)(3r-2\bar{\lambda}\mu-\bar{\mu}\tau)\tau+r^3(1-|\lambda|^2)(1-|\mu|^2)(1-|\tau|^2)\alpha].\\
&=c_4(r,s,\lambda,\mu,\tau)+\rho_4(r,s,\lambda,\mu,\tau)\alpha.
\end{align*}

Now $\alpha \in \ov{\D}$ is arbitrary, so the closed disk $\overline{\D}(c_4(r,s,\lambda,\mu,\tau),\rho_4(r,s,\lambda,\mu,\tau))$ is covered.

We know that $f^{(4)}(r)\in \partial \D(c_4(r,s,\lambda,\mu,\tau), \rho_4(r,s,\lambda,\mu,\tau))$ if and only if $f(z)=zg(z)$, where $g$ is a Blaschke product of degree 4 satisfying \eqref{condition2}, and then
we apply this fact to determine the precise form of $g$. Set
$$h(z)=\dfrac{T_{-u_0}\circ g \circ T_{r}(z)}{z},\quad z\in \D.$$
Clearly, $h$ is a Blaschke product of degree 3 depending on $g$ and satisfying
$$h(0)=(T_{-u_0}\circ g \circ T_r)'(0)=\lambda,$$
$$h'(0)=\frac{(T_{-u_0}\circ g \circ T_r)''(0)}{2}=\mu(1-|\lambda|^2).$$
and
$$h''(0)=\frac{(T_{-u_0}\circ g \circ T_r)'''(0)}{3}=2(1-|\lambda|^2)\left[-\overline{\lambda}\mu^2+\tau(1-|\mu|^2)\right].$$
Then $H(z)=T_{-\lambda}\circ h(z)$ is a Blaschke product of degree 3 fixing $0$.  Set
$$G(z)=\frac{H(z)}{z}.$$
Obviously, $G$ is a Blaschke product of degree 2 depending on $g$ and satisfying
$$G(0)=H'(0)=T_{-\lambda}'(\lambda) h'(0)=\mu.$$
$$G'(0)=\frac{(T_{-\lambda}\circ h)''(0)}{2}=\tau(1-|\mu|^2).$$
Thus $T_{-\mu}\circ G$ is a Blaschke product of degree 2 fixing 0, set
$$F(z)=\frac{T_{-\mu}\circ G(z)}{z},$$
then $F(z)$
is an automorphism of $\D$ depending on $g$ and satisfying
$$F(0)=(T_{-\mu}\circ G)'(0)=T_{-\mu}'(\mu) G'(0)=\tau.$$
Thus $T_{-\tau}\circ F$ is an automorphism of $\D$ fixing 0,
which means that $T_{-\tau}\circ F(z)=e^{i \theta} z$ for $z\in \D$ and $\theta\in \R$.
Now it is easy to check that
$$g(z)=T_{u_0}\left(T_{-r}(z) T_{\lambda}(T_{-r}(z) T_{\mu}(T_{-r}(z) T_{\tau}(e^{i \theta} T_{-r}(z))))\right), \quad z\in \D.$$
Conversely, if
$f(z)=z T_{u_0}\left(T_{-r}(z) T_{\lambda}(T_{-r}(z) T_{\mu}(T_{-r}(z) T_{\tau}(e^{i \theta} T_{-r}(z))))\right)$, where $\th \in \R$, then direct calculations gives
$$f^{(4)}(r)=c_4(r,s,\lambda,\mu,\tau)+\rho_4(r,s,\lambda,\mu,\tau) e^{i \theta}\in \partial \D(c_4(r,s,\lambda,\mu,\tau), \rho_4(r,s,\lambda,\mu,\tau)).$$
Hence we complete the proof.
\end{proof}
\section{Variability region for the fourth derivative}
Let $\beta_1,\beta_2\in \D$, we analyze  the variability region
$$V(z_0,w_0,\beta_1,\beta_2)=\{f^{(4)}(z_0):f\in \H_0(z_0,w_0,\beta_1,\beta_2)\},$$
where
\begin{align*}
&\H_0 (z_0,w_0,\beta_1) =  \{ f \in \H_0 : f(z_0) =w_0, f'(z_0)=c_1(z_0,w_0)+\rho_1(z_0,w_0)\dfrac{r\beta_1}{z_0}\},\\
&\H_0 (z_0,w_0,\beta_1,\beta_2)\! = \! \{ f \in \H_0(z_0,w_0,\beta_1) \!: \! f''(z_0)\!=\!c_2(z_0,w_0,\beta_1)\!+\!\rho_2(z_0,w_0,\beta_1)\frac{r\beta_2}{z_0}\}.
\end{align*}

Since the relation
$V(r,s,\lambda,\mu)=e^{i(4\vp-\xi)}V(z_0,w_0,\beta_1,\beta_2)$ holds for $\lambda=e^{-i \xi}\beta_1$ and $\mu=e^{i(\vp-\xi)}\beta_2$, where $z_0=re^{i \vp}, w_0=s e^{i \xi}\in \D$ with $s<r$,
it is sufficient to determine the variability region $V(r,s,\lambda,\mu )$, $\lambda,\mu\in \D$. 

We define $c(\zeta),\rho(\zeta)$ and $V$ by
\begin{equation}\label{eq:c(zeta)-rho(zeta)}
c(\zeta)=\zeta (1-\eta\zeta),\quad \rho(\zeta)=t(1-|\zeta|^2),\quad V
 = \bigcup_{\zeta \in \overline{\mathbb{D}}}
  \overline{\mathbb{D}}(c( \zeta), \rho(\zeta )),
\end{equation}
where
$$\eta=\frac{r\bar{\mu}}{1+2r^2-2s\lambda-2r\bar{\lambda}\mu},\quad  t=\frac{r}{|1+2r^2-2s\lambda-2r\bar{\lambda}\mu|}.$$
Then by the fourth-order Dieudonn\'e Lemma, we have
$$V(r,s,\lambda ,\mu) = A
  \left(B+C V\right),$$
  where $C\overline{\mathbb{D}}(c, \rho)$ means
$\overline{\mathbb{D}}(Cc, |C|\rho)$ and
\begin{equation}\label{eq:C}
C=r^2(1-|\lambda|^2)(1-|\mu|^2)(1+2r^2-2s\lambda-2r\bar{\lambda}\mu)\in \C.
\end{equation}

Since the set $V$ has the same properties as $V(r,s,\lambda,\mu )$, we just need to determine the set $V$, which is reduced to the case in \cite{chen2021}.
Therefore, we can immediately obtain the following theorem, analogous to \cite[Theorem 3.3]{chen2021}, which gives the parametric representation of $\p V(r,s,\lambda,\mu)$. Recall that $A,B,C$ are given in \eqref{eq:AB} and \eqref{eq:C}.
\begin{thm}
\label{thm:boundary-curve}
Let $0 \leq s < r< 1$ and $|\lambda|<1$, $|\mu|<1$. For $\theta \in \mathbb{R}$, let
$t_\theta$ be the unique solution to the
equation
\begin{equation}
\label{eq:definig_equation_of_r_theta}
   |xe^{i\theta} - \ov{\eta}| = 2(x^2-|\eta|^2) , \quad x >|\eta| ,
\end{equation}
if $|xe^{i\theta} -\ov{\eta}| \geq  2(x^2-|\eta|^2) $;
otherwise let $t_\theta =t$.

Set
\begin{equation}
\label{eq:def_of_zeta_theta}
  \zeta_\theta = \frac{t_\theta e^{i\theta} -\ov{\eta} }{2(t_\theta^2-|\eta|^2)}
   \in \overline{\mathbb{D}} .
\end{equation}
Then $V(r,s,\lambda,\mu)$ is a convex closed domain enclosed by the Jordan curve $\partial V(r,s,\lambda,\mu)$ and the parametric representation
$(- \pi , \pi ] \ni \theta \mapsto \gamma ( \theta )$
of $\partial V(r,s,\lambda,\mu)$
 is given as follows.
\begin{enumerate}
 \item[{\rm (i)}]
If $t+|\eta|\leq \frac{1}{2} $, then $|te^{i \theta}-\ov{\eta}| \geq 2(t^2-|\eta|^2)$
for all $\theta \in \mathbb{R}$ and
$$
   \gamma( \theta )=
 A \left(B+C c( \zeta_\theta ) \right)   \in \partial V(r,s,\lambda,\mu).$$
  \item[{\rm (ii)}]
If $t-|\eta|\leq \frac{1}{2} $, then $|te^{i \theta}-\ov{\eta}| \le 2(t^2-|\eta|^2)$
for all $\theta \in \mathbb{R}$ and
 $$\gamma( \theta ) =
    A \left(B+C( c( \zeta_\theta )  + \rho(\zeta_\theta) e^{i \theta}) \right)   \in \partial V(r,s,\lambda,\mu).$$
 \item[{\rm (iii)}]
If  $t+|\eta| > \frac{1}{2}$ and $t-|\eta| < \frac{1}{2} $, then
\[
   \gamma( \theta ) =
  \begin{cases}
  A \left(B+C( c( \zeta_\theta )  + \rho(\zeta_\theta) e^{i \theta}) \right) , \quad
   & |te^{i \theta}-\ov{\eta}| < 2(t^2-|\eta|^2) , \\[2ex]
      A \left(B+C c( \zeta_\theta ) \right)  , \quad
   & |te^{i \theta}-\ov{\eta}|\ge 2(t^2-|\eta|^2).
  \end{cases}	
\]
\end{enumerate}
\end{thm}
\begin{rmk}
We can explicitly determine all the extremal functions $f\in \H_0(r,s,\lambda,\mu)$ with $f^{(4)}(r) \in \p V(r,s,\lambda,\mu)$.
The equality
\[
    f^{(4)}(r) =  A \left(B+C( c( \zeta_\theta )  + \rho(\zeta_\theta) e^{i \theta} )\right)
\]
holds for some $\theta \in \mathbb{R}$ with $\zeta_\theta \in \mathbb{D}$
if and only if
$$
f(z)=z T_{\frac{s}{r}}\left(T_{-r}(z) T_{\lambda}(T_{-r}(z) T_{\mu}(T_{-r}(z) T_{\zth}(e^{i (\theta+\arg C)} T_{-r}(z))))\right), \quad z \in \mathbb{D}.$$
Similarly the equality
\[
    f^{(4)}(r) =   A \left(B+C c (\zeta_\theta) \right)
\]
holds for some $\theta \in \mathbb{R}$ with $\zeta_\theta \in \partial \mathbb{D}$
if and only if
$$
  f(z)=
    z T_{\frac{s}{r}}\left(T_{-r}(z) T_{\lambda}(T_{-r}(z) T_{\mu}(\zth T_{-r}(z)))\right), \quad z \in \mathbb{D}.
$$
\end{rmk}
We end this section by asking the interesting question: is it possible to explicitly determine the variability region
$\{f^{(4)}(z_0): f\in \H_0,f(z_0) =w_0\}$ for given $z_0,w_0\in \D$ with $|w_0|<|z_0|$?
\bibliographystyle{amsplain} 


\end{document}